\documentclass[a4paper,11pt]{article}
\usepackage{pdfsync}

 \usepackage[plainpages=false]{hyperref}
 \usepackage{amsfonts,amsmath,amssymb,amsthm,enumerate}
 \usepackage{latexsym,lscape,rawfonts,mathrsfs,mathtools}
\usepackage{enumitem}
\usepackage{relsize}
\makeatletter
\newcommand*\owedge{\mathpalette\@owedge\relax}
\newcommand*\@owedge[1]{%
  \mathbin{%
    \ooalign{%
      $#1\m@th\bigcirc$\cr
      \hidewidth$#1\m@th\wedge$\hidewidth\cr
    }%
  }%
}
\makeatother

 \usepackage[all]{xy}
 \usepackage{makeidx}
 \usepackage{graphicx,psfrag}
 \usepackage{pstool}
 \usepackage{tikz-cd}
 \usepackage{array,tabularx,tabu}

 \usepackage{setspace}

 \usepackage[titletoc,title]{appendix}

\usepackage{txfonts}

 \newcommand{\C}{\ensuremath{\mathbb{C}}}

 \newcommand{\R}{\ensuremath{\mathbb{R}}}

 \newcommand{\RP}{\ensuremath{\mathbb{RP}}}

 \newcommand{\ba}{\begin{align*}}
 \newcommand{\ea}{\end{align*}}
 \newcommand{\na}{\nabla}
\newcommand{\la}{\langle}
\newcommand{\ra}{\rangle}
\newcommand{\Rc}{\mathrm{Rc}}
\newcommand{\Rm}{\mathrm{Rm}}
\newcommand{\lc}{\left(}

\newcommand{\rc}{\right)}
\newcommand{\ep}{\epsilon}
\newcommand{\tl}{\left(\partial_t-\Delta \right)}
\newcommand{\ka}{K\"ahler\,}

\newcommand{\raa}[1]{\kern-1.5ex\xrightarrow{\ \ #1\ \ }\phantom{}\kern-1.5ex}
\newcommand{\ras}[1]{\kern-1.5ex\xrightarrow{\ \ \smash{#1}\ \ }\phantom{}\kern-1.5ex}
\newcommand{\da}[1]{\bigg\downarrow\raise.5ex\rlap{\scriptstyle#1}}

 \makeatletter
 \def\ExtendSymbol#1#2#3#4#5{\ext@arrow 0099{\arrowfill@#1#2#3}{#4}{#5}}
 
 \makeatother

 \makeatletter
 \def\ExtendSymbol#1#2#3#4#5{\ext@arrow 0099{\arrowfill@#1#2#3}{#4}{#5}}
 
 \makeatother

\def\XXint#1#2#3{{\setbox0=\hbox{$#1{#2#3}{\int}$ }
\vcenter{\hbox{$#2#3$ }}\kern-.55\wd0}}

\numberwithin{equation}{section}

\newtheorem{thm}{Theorem}[section]
\newtheorem{cor}[thm]{Corollary}
\newtheorem{prop}[thm]{Proposition}
\newtheorem{lem}[thm]{Lemma}
\newtheorem{conj}[thm]{Conjecture}

\newtheorem{rem}[thm]{Remark}

\newtheorem{defn}[thm]{Definition}

\title{Ancient solutions to the K\"ahler Ricci flow}
\author{Yu Li}
\date{\today}

\begin{document}
\maketitle

\begin{abstract}
We give a complete classification of all $\kappa$-noncollapsed, complete ancient solutions to the K\"ahler Ricci flow with nonnegative bisectional curvature.
\end{abstract}


\section{Introduction}

In this paper, we study the $\kappa$-solutions to the Ricci flow. Recall that a $\kappa$-solution $(M^n,g(t))$ is an ancient solution to the Ricci flow
\begin{align*}
\frac{\partial}{\partial t} g(t)=-2\Rc(g(t))
\end{align*}
defined for $t \in (-\infty,0]$ such that for each $t$, $(M^n,g(t))$ is a $\kappa$-noncollapsed, complete Riemannian manifold with weakly PIC$_2$. In addition, the curvature operator $\Rm$ is uniformly bounded on the spacetime $M \times (-\infty,0]$. For a precise definition of $\kappa$-solutions, see Definition \ref{def:flow1}.

The concept of $\kappa$-solutions plays a vital role in the study of the singularity in the Ricci flow. In Perelman's celebrated work on the Poincar\'e conjecture and the geometrization conjecture \cite{Pe1}\cite{Pe2}\cite{Pe3}, it is proved that the high curvature regions of the Ricci flow on closed $3$-manifolds are modeled on $\kappa$-solutions. By investigating the geometric and topological properties of the $3$-dimensional $\kappa$-solutions, Perelman proved a canonical neighborhood theorem for the high curvature part and was able to continue the Ricci flow by performing the surgery.

In general, any ancient solution from a finite time singularity of a compact Ricci flow is $\kappa$-noncollapsed by the monotonicity formula of Perelman's entropy \cite{Pe1}. Therefore, to investigate the ancient solution, it is natural to add the $\kappa$-noncollapsing condition, which excludes the complicated local collapsing phenomenon at spatial or time infinity.

In dimension $3$, it follows from the famous Hamilton-Ivey pinching that any ancient solution must have nonnegative sectional curvature. For the higher dimensional cases, Chen \cite{Chen09} proved that the scalar curvature of any ancient solution is nonnegative. However, nonnegative scalar curvature alone is usually too weak to do any further analysis. Therefore, in many articles, various positivity conditions for the curvature are imposed.

The curvature assumption we consider here is the weakly PIC$_2$ condition, which was introduced by Brendle-Schoen in \cite{BS09}. In the $3$-dimensional case, the concept of weakly PIC$_2$ is equivalent to the nonnegativity of the curvature operator. However, the former is weaker than the latter in the higher dimensional case and hence the $\kappa$-solution defined here has a weaker assumption than that defined by Perelman \cite{Pe1}. Many important properties of $\kappa$-solutions still hold under the weakly PIC$_2$ assumption, see \cite{Bre09}\cite{Bre10b}\cite{CW15}, etc.

In dimension $2$, all $\kappa$-solutions are $\R^2$, $S^2$ and $\RP^2$, proved by Perelman \cite{Pe1}. In recent breakthrough works, Brendle \cite{Bre18} in the noncompact case and Brendle, Daskalopoulos and Sesum \cite{BDS20} in the compact case gave the complete classification of all $3$-dimensional $\kappa$-solutions. More precisely, they show that shrinking cylinders, the Bryant soliton, shrinking spheres, Perelman's solution, and their quotients are all possibilities. The classification of higher dimensional $\kappa$-solutions remains open. For recent developments in higher dimensions, see for example Brendle, Huisken and Sinestrari \cite{BHS11}, Li and Zhang \cite{LZ18}, Brendle and Naff\cite{BN20} and Cho and Li \cite{CL20}.

The analogous question is to classify all $\kappa$-solutions to the K\"ahler Ricci flow, which by definition are $\kappa$-noncollapsed, complete ancient solutions to the K\"ahler Ricci flow with bounded and nonnegative bisectional curvature (Definition \ref{def:flow 2}). Notice that the weakly PIC$_2$ condition is replaced by nonnegative bisectional curvature, which is weaker and more natural in the K\"ahler setting. In the compact case, all $\kappa$-solutions to the K\"ahler Ricci flow are classified in \cite{DZ20b}.

The main result of the paper is the following classification theorem for the general case.

\begin{thm} \label{thm:main1}
Let $(M^m,g(t))_{t \in (-\infty,0]}$ be a $\kappa$-noncollapsed, complete ancient solution to the K\"ahler Ricci flow with nonnegative bisectional curvature. Then it is isometrically biholomorphic to $N^k \times \C^{m-k}$, where $N$ is a compact Hermitian symmetric space.
\end{thm}

Theorem \ref{thm:main1} extends our earlier result regarding the classification of $\kappa$-solutions on K\"ahler surfaces \cite[Theorem $1.3$]{CL20}. In \cite{CL20}, a key observation is that any complete, complex $2$-dimensional ancient solution to the K\"ahler Ricci flow with nonnegative bisectional curvature automatically has nonnegative curvature operator \cite[Lemma $4.6$]{CL20}. Here, we obtain a similar curvature improvement (see Theorem \ref{thm:pic2}), which states that any complete ancient solution to the K\"ahler Ricci flow with nonnegative bisectional curvature has weakly PIC$_2$. Therefore, any $\kappa$-solution to the K\"ahler Ricci flow is also a $\kappa$-solution to the (real) Ricci flow.

Another important observation is that any $\kappa$-solution to the K\"ahler Ricci flow must be of Type-I (see Lemma \ref{lem:type1}), which was initially proved by Deng-Zhu \cite[Lemma $2.5$]{DZ20b} in the compact case. Essentially, this follows from the classical point-picking argument and the nonexistence of a $\kappa$-noncollapsed, nonflat, steady K\"ahler Ricci soliton with nonnegative bisectional curvature \cite{DZ20a}. Here, the K\"ahler condition is crucial, and one cannot expect the same conclusion to hold for the real case (e.g., the Bryant soliton).

Therefore, to prove Theorem \ref{thm:main1}, it is essential to understand all Type-I $\kappa$-solutions to the Ricci flow. One motivation is the similar classification result of Ricci shrinkers \cite{MW17}\cite{LN20}. More precisely, it was proved in \cite[Theorem $3.1$(ii)]{LN20} that any Ricci shrinker with weakly PIC$_2$ is locally symmetric. On the one hand, Li-Wang proved in \cite{LW19} that any Ricci shrinker is $\kappa$-noncollapsed, a result previously established by Naber \cite{Na10} for Ricci shrinkers with bounded curvature. On the other hand, it is well-known that any Ricci shrinker with bounded curvature can be regarded as a self-similar, Type-I, ancient solution to the Ricci flow. Therefore, it is natural to expect that any Type-I $\kappa$-solution is also locally symmetric. We confirm this by the following theorem, see also \cite{Zh18} and \cite{Hal19} for the $3$-dimensional case.

\begin{thm} \label{thm:main2}
Let $(M^n,g(t))_{t \in (-\infty,0]}$ be a $\kappa$-noncollapsed, Type-I, complete ancient solution to the Ricci flow with weakly \emph{PIC$_2$}. Then it is isometric to a finite quotient of $N^k \times \R^{n-k}$, where $N$ is a simply connected compact symmetric space.
\end{thm}

Theorem \ref{thm:main1} follows from Theorem \ref{thm:main2} immediately, provided that the curvature is uniformly bounded (Theorem \ref{thm:ka}). To complete the proof of Theorem \ref{thm:main1}, we next show that any solution in Theorem \ref{thm:main1} must have bounded curvature. Here, we follow the strategy in \cite{CL20} to prove first a canonical neighborhood theorem, which states that any high curvature region in the flow is modeled on a $\kappa$-solution to the K\"ahler Ricci flow (Theorem \ref{thm:cano2}). The key point is to show that any K\"ahler manifold, whose high curvature regions are modeled on $\kappa$-solutions, must have uniformly bounded curvature, see Proposition \ref{prop:cano}. Notice that in the $3$-dimensional case, a similar conclusion is proved by the fact that any complete manifold with nonnegative sectional curvature cannot have smaller and smaller $\ep$-necks \cite[Proposition $2.2$]{CZ06}. In our case, we show that the high curvature part locally has a singular fibration structure on an orbifold, where each fiber is a simply connected, compact, Hermitian symmetric space $N$ or its quotient. By analyzing the K\"ahler form, we show that the volumes of any two regular fibers are uniformly comparable and hence a global bound for the curvature can be obtained, see Proposition \ref{prop:cano}. With the canonical neighborhood theorem, the rest of the proof follows from similar arguments in \cite{CL20} by showing the curvature is bounded first for each time slice and then for any compact time interval. Therefore, the global curvature bound follows from the Harnack inequality.

Now, we briefly discuss the proof of Theorem \ref{thm:main2}. Notice that the conclusion on compact manifolds has been proved by Ni \cite{N09}, under the strictly PIC$_2$ condition. By using similar arguments, one can easily prove Theorem \ref{thm:main2} on compact manifolds; see Theorem \ref{thm:cpt}. For the noncompact case, we analyze the asymptotic Ricci shrinker, which is helpful to understand the geometric behavior as $t \to -\infty$. For this purpose, we consider the moduli space $\mathcal M$, see Definition \ref{def:modu}, such that any asymptotic Ricci shrinker of a Type-I $\kappa$-solution belongs to $\mathcal M$. 

To prove Theorem \ref{thm:main2}, we may assume the $\kappa$-solution $(M^n,g(t))$ is simply connected and irreducible, by Lemma \ref{lem:quo} and an inductive argument. If the solution is nonflat, then we can derive a contradiction. Indeed, we fix a point $p$ and a time sequence $t_i=-4^i$ such that the geometry on $B_i=B_{g(t_i)}(p, \ep^{-1} \sqrt{|t_i|})$ is modeled on an element in $\mathcal M$, if $i$ is sufficiently large. By a careful analysis of $\mathcal M$ and the structure of $(M^n,g(t))$, we prove that all $B_i$ are close to a fixed $N^k \times \R^{n-k}$. The difficulty here is to prove the uniqueness of the model space and exclude any possible quotient. In particular, there exists a local fibration structure on $B_i$. It is clear from the distance comparison, see Theorem \ref{lem:asym}(i), that $\{B_i\}$ forms an exhaustion of $M$. Therefore, a standard argument yields a global fibration on $M$ with fiber $N$. On the other hand, $M$ is known to be diffeomorphic to $\R^n$, see \cite[Theorem $1.2$]{CW15}. Therefore, we obtain a contradiction since $\R^n$ does not admit a fibration with nontrivial compact fibers \cite{BS50}.

From Theorem \ref{thm:main2}, we know that any Type-I $\kappa$-solution is essentially a Ricci shrinker. Notice that for $\kappa$-solutions, one can rescale each factor in the decomposition of $N$, but for Ricci shrinkers, the scaling is fixed so that each factor has the same Einstein constant. For Type-II $\kappa$-solutions, the classification is much more difficult. It is still unclear how a Type-II $\kappa$-solution is related to a steady soliton.

This paper is organized as follows. In Section 2, we review some basic properties of Type-I $\kappa$-solutions and Ricci shrinkers. Then we discuss the structure of the singular fibration and prove Theorem \ref{thm:main2}. In Section 3, we obtain the curvature improvement for the K\"ahler Ricci flow and prove Theorem \ref{thm:main1}. In Section 4, we propose a conjecture for general Type-I ancient solutions.
\\
\\
\textbf{Acknowledgements:} The author would like to thank Prof. Xiuxiong Chen, Prof. Bing Wang and Prof. Simon Brendle for helpful discussions and comments. The author is supported by YSBR-001, NSFC-12201597 and research funds from the University of Science and Technology of China and the Chinese Academy of Sciences.

\section{Classification of Type-I $\kappa$-solutions}

\begin{defn}\label{def:flow1} Let $(M^n,g(t))_{t \in (-\infty,0]}$ be a complete ancient solution to the Ricci flow.
\begin{enumerate}[label=(\roman*)]
\item $(M^n,g(t))_{t \in (-\infty,0]}$ is of Type-I if there exists a constant $C_0>0$ such that
\begin{align}
|\emph{Rm}|(x,t) \le \frac{C_0}{1+|t|} \label{E200}
\end{align}
for any $(x,t) \in M \times (-\infty,0]$.

\item Given $\kappa>0$, we say that $(M^n ,g(t))_{t \in (-\infty,0]}$ is $\kappa$-noncollapsed if for any ball $B_{g(t)}(x,r)$ satisfying $R(y,t) \le r^{-2}$ for all $ y \in B_{g(t)}(x,r)$, we have
\begin{align*}
|B_{g(t)}(x,r)|_{g(t)}\ge \kappa r^n.
\end{align*}

\item $(M^n,g(t))_{t \in (-\infty,0]}$ has weakly \emph{PIC$_2$} if for any $(x,t) \in M \times (-\infty,0]$ and any $\zeta,\eta \in T_x M \otimes \C$, we have at $(x,t)$,
\begin{align*}
\emph{Rm}(\zeta,\eta,\bar \zeta,\bar \eta) \ge 0. 
\end{align*}
For other equivalent definitions of weakly \emph{PIC$_2$}, see \emph{\cite[Chapter $7$]{Bre10b}}.

\item \emph{($\kappa$-solution to the Ricci flow)}
$(M^n,g(t))_{t \in (-\infty,0]}$ is called a $\kappa$-solution if it has weakly \emph{PIC$_2$}, uniformly bounded curvature and is $\kappa$-noncollapsed.
\end{enumerate}
\end{defn}

Next, we recall the following definition of Ricci shrinkers.
\begin{defn} \emph{(Ricci shrinker)}
A Ricci shrinker $(M^n,g,f)$ is a complete Riemannian manifold equipped with a smooth potential function $f$ satisfying
\begin{align*} 
\Rc+\emph{Hess}\,f=\frac{1}{2}g,
\end{align*}
where $f$ is normalized so that
\begin{align*} 
R+|\nabla f|^2=f. 
\end{align*}
\end{defn}
Any Ricci shrinker can be regarded as a self-similar ancient solution to the Ricci flow, see \cite[Chapter $4$]{CLN06}. It is proved in \cite{LW19} that any Ricci shrinker is $\kappa$-noncollapsed for some $\kappa>0$. Moreover, we have the following classification of Ricci shrinkers with weakly PIC$_2$ from \cite{LN20}, see also \cite[Corollary $4$]{MW17}.

\begin{lem} [Theorem $3.1$ (ii) of \cite{LN20}]\label{lem:shrinker}
Let $(M^n,g,f)$ be a Ricci shrinker with weakly \emph{PIC$_2$}. Then $(M^n,g)$ is isometric to a finite quotient of $N^k \times \R^{n-k}$, where $N$ is a simply connected compact symmetric space.
\end{lem}

For Ricci shrinker $(N^k \times \R^{n-k},\bar g)$, the metric $\bar g=g_N \times g_E$, where $g_N$ is an Einstein metric with Einstein constant $1/2$ and $g_E$ is the Euclidean metric. If $n-k \ge 1$, then the potential function $f$ of the Ricci shrinker is $f=\frac{|z|^2}{4}+\frac{k}{2}$, where $z \in \R^{n-k}$. If $n=k$, then the Ricci shrinker is compact and $f=\frac n 2$. In the following, we will use $\bar g$ to denote both the metric of the Ricci shrinker $N^k \times \R^{n-k}$ and the induced metric on the quotient $(N^k \times \R^{n-k})/\Gamma$. Also, for any Ricci shrinker $(N^k \times \R^{n-k})/\Gamma$, the associated Ricci flow $((N^k \times \R^{n-k})/\Gamma, \bar g(t))_{t\in (-\infty,0)}$, where $\bar g(-1)=\bar g$, is implicitly understood.

Next, we recall that the fundamental group of any Ricci shrinker is always finite; see \cite{Lott03} \cite{Zhang07} \cite{FLGR08} \cite{WW07}, etc. In fact, the order can be controlled by the $\kappa$-noncollapsing condition; see \cite[Proposition 3]{CL16}. More precisely,

\begin{lem} \label{lem:fund1}
For any $\kappa$-noncollapsed Ricci shrinker $(M^n,g,f)$, there exists a positive constant $C=C(n,\kappa)<\infty$ such that the order of the fundamental group satisfies
\begin{align*} 
|\pi_1(M) |\le C.
\end{align*}
\end{lem} 

Now, we define the following moduli space.
\begin{defn} \label{def:modu}
The moduli space $\mathcal M(n,\kappa)$ consists of all $\kappa$-noncollapsed, nonflat, Ricci shrinkers in the form of $(N^k \times \R^{n-k})/\Gamma$, where $N$ is a simply connected compact symmetric space and $\Gamma \subset \emph{Iso}(N^k \times \R^{n-k})$ is a finite subgroup which acts freely on $N^k \times \R^{n-k}$.
\end{defn}
It follows from Lemma \ref{lem:fund1} that for any $(N^k \times \R^{n-k})/\Gamma \in \mathcal M(n,\kappa)$,
\begin{align} 
|\Gamma |\le C(n,\kappa).
\label{E202a}
\end{align}

\begin{lem} \label{lem:finite}
There are finitely many isometry classes in the moduli space $\mathcal M(n,\kappa)$.
\end{lem}

\begin{proof}
For any simply connected compact symmetric space $N$, we have by de Rham's decomposition theorem \cite[Chapter IV, Theorem $6.2$]{KoNo63}
\begin{align*} 
N=N_1 \times N_2 \times \cdots \times N_l,
\end{align*}
where each $N_i$ is a simply connected, irreducible, compact symmetric space. Moreover, each factor $N_i$ has the same Einstein constant $1/2$. From the classification, see \cite[Chapter $7$]{Be08}, of all simply connected, irreducible, compact symmetric spaces, there are finitely many possible $N$ with dimension not greater than $n$ and Einstein constant $1/2$.

We fix a space $N^k \times \R^{n-k}\in \mathcal M(n,\kappa)$ and consider all its possible quotient $(N^k \times \R^{n-k})/\Gamma$. For any $\sigma \in \Gamma$, we can write $\sigma=(\sigma_1,\sigma_2) \in \text{Iso}(N^k) \times \text{Iso}(\R^{n-k})$ from the uniqueness of de Rham's decomposition theorem. If we denote the projections of $\Gamma$ to $\text{Iso}(N^k)$ and $\text{Iso}(\R^{n-k})$ by $\Gamma_0$ and $\Gamma_1$, respectively, then both $\Gamma_0$ and $\Gamma_1$ are finite groups. From the center of mass construction, there exists a point $x \in \R^{n-k}$ fixed by $\Gamma_1$. Therefore, by a translation, we can regard
\begin{align*} 
\Gamma \subset \text{Iso}(N^k) \times O(n-k).
\end{align*}
Since $\text{Iso}(N^k) \times O(n-k)$ is a compact group, there are only finitely many conjugacy classes of any finite subgroup with fixed order; see, for example, \cite[Corollary 8.1.7]{BK81}. From \eqref{E202a}, it is easy to see that there are finitely many quotients $(N^k \times \R^{n-k})/\Gamma$.
\end{proof}

For any $(N^k \times \R^{n-k})/\Gamma \in \mathcal M(n,\kappa)$, we assume $n-k \ge 1$ and analyze its fibration structure. As above, we denote the projection of $\Gamma$ to $\text{Iso}(\R^{n-k})$ by $\Gamma_1$. Therefore, one can regard $(N^k \times \R^{n-k})/\Gamma$ as a singular fibration over the orbifold $\R^{n-k}/\Gamma_1$ such that each fiber is a quotient of $N$. More precisely, the fibration is induced by the natural projection map
\begin{align*} 
(N^k \times \R^{n-k})/\Gamma \overset{\pi}{\longrightarrow} \R^{n-k}/\Gamma_1.
\end{align*}
For any connected open set $U \subset \R^{n-k}/\Gamma_1$, there exist a connected open set $\tilde U \subset \R^{n-k}$, which projects onto $U$, and a subgroup $G \subset \Gamma_1$ such that $\tilde U$ is invariant under the action of $G$. Moreover, $U$ is isometric to $\tilde U/G$. It is clear that the following diagram is commutative:
\begin{center}
\begin{tikzcd}
N \times \tilde U \arrow[r] \arrow[d, "p_2"] & \pi^{-1}(U) \arrow[d, "\pi"] \\
\tilde U \arrow[r,"\text{pr}"] & U 
\end{tikzcd}
\end{center}
We define $\Lambda$ to be the preimage of $G$ for the projection $\Gamma \to \Gamma_1$. Then the projection map $p_2$ in the above diagram is $(\Lambda,G)$-equivariant. Moreover, it is clear that $(N \times \tilde U)/\Lambda$ is isometric to $\pi^{-1}(U)$. Notice that for any $x \in (N^k \times \R^{n-k})/\Gamma$, there is a natural decomposition of the tangent space at $x$ into horizontal and vertical directions from the fibration structure.

For any $p \in \R^{n-k}/\Gamma_1$, the fiber $\pi^{-1}(p)$ is a nontrivial quotient of $N$ if $p$ is a singular point. Moreover, there exists a small constant $r>0$ such that we can choose $U=B_d(p,r)$ and $\tilde U=B_{g_E}( \tilde p, r)$, where $\tilde p$ is a lift of $p$ and $d$ is the induced distance in $\R^{n-k}/\Gamma_1$. In this case, the group $G=G_{\tilde p}$ is the isotropy group of $\tilde p$ in $\Gamma_1$. We call the open set $B_{g_E}(\tilde p, r)/G_{\tilde p}$ an orbifold chart and the open set $(N \times B_{g_E}(\tilde p, r))/\Lambda$ a fundamental chart.

For later applications, we construct a family of fundamental charts for any space in $\mathcal M(n,\kappa)$.

\begin{lem} \label{lem:chart}
For any $n$ and $\kappa>0$, there exists a constant $\delta_0=\delta_0(n,\kappa) \in (0,1)$ satisfying the following property.

For any $\bar r>0$ and $(N^k \times \R^{n-k})/\Gamma \in \mathcal M(n,\kappa)$ with $n-k \ge 1$, there are countably many fundamental charts $(N \times B_{g_E}(\tilde p_i, 100 r_i))/\Lambda_i$ such that $r_i \in [\delta_0 \bar r , \bar r]$ and
\begin{align*} 
(N^k \times \R^{n-k})/\Gamma =\bigcup_{i \ge 1} (N \times B_{g_E}(\tilde p_i,r_i))/\Lambda_i.
\end{align*}
\end{lem}

\begin{proof}
From Lemma \ref{lem:finite}, we only need to prove the lemma for a fixed space $(N^k \times \R^{n-k})/\Gamma \in \mathcal M(n,\kappa)$. As before, we set $\Gamma_1$ to be the projection of $\Gamma$ to $\text{Iso}(\R^L)$, where $L:=n-k$. Since $\Gamma_1$ is finite, we may assume $\Gamma_1 \subset O(L)$ by a translation. Therefore, the orbifold $\R^L/\Gamma_1$ can be regarded as a cone over the compact orbifold $S^{L-1}/\Gamma_1$ and hence we may assume $\bar r=1$ by rescaling.

\textbf{Claim}: If $\delta$ is sufficiently small, any point $x \in \R^L$ is contained in a ball $B_{g_E}(\tilde p,r)$ with $r \in [\delta , 1]$ such that $B_{g_E}(\tilde p, 100 r)/G_{\tilde p}$ is an orbifold chart.

Suppose the claim does not hold. Then there exists a sequence $x_i \in \R^L$ such that $x_i$ is not contained in any open set in the form of $B_{g_E}(\tilde p, r)$ with $r \in [\delta_i,1]$, where $B_{g_E}(\tilde p, 100 r)/G_{\tilde p}$ is an orbifold chart and $\delta_i \to 0$.

Notice that $x_i \to \infty$. Otherwise, we have $\lim_{i \to \infty} x_i=x$ by taking a subsequence and $B_{g_E}(x, r_0)/G_{x}$ is an orbifold chart for a small constant $r_0>0$, which contradicts our assumption. Moreover, there exists $\gamma_i \in \Gamma_1$ such that 
\begin{align} 
\gamma_i(x_i) \ne x_i \quad \text{and}\quad \lim_{i\to \infty}|\gamma_i (x_i)-x_i|=0 \label{EX100a}
\end{align}
since otherwise, by taking a subsequence, $B_{g_E}(x_i, r_1)/G_{x_i}$ is an orbifold chart for a small constant $r_1>0$. Since $\Gamma_1$ is finite, we assume all $\gamma_i$ are the same and denote it by $\gamma$. By taking a subsequence if necessary, we assume
\begin{align*} 
\lim_{i \to \infty} \frac{x_i}{|x_i|}=y \in S^{L-1}
\end{align*}
and hence by \eqref{EX100a} and the fact that $x_i \to \infty$,
\begin{align*} 
\gamma(y)=y. 
\end{align*}
There exists a number $s_0>0$ such that $B_{g_E}(y, s_0)/G_{y}$ is an orbifold chart. If we set $ y_i=|x_i| y$, then it is clear that $y_i$ is also fixed by $\gamma$ and $G_{y_i}=G_y$. Therefore, $B_{g_E}(y_i, ts_0)/G_{y_i}$ is an orbifold chart for any $1 \le t \le |x_i|$. By our assumption, we conclude that
\begin{align*} 
\lim_{i \to \infty} |x_i-y_i|=\infty \quad \text{and} \quad \lim_{i \to \infty} \frac{|x_i-y_i|}{|x_i|}=0. 
\end{align*}
Next, by taking a subsequence, we assume
\begin{align*} 
\lim_{i \to \infty} \frac{x_i-y_i}{|x_i-y_i|}=z \in S^{L-1}.
\end{align*}
Then it is clear that $z$ is perpendicular to $y$ and $\gamma(z)=z$. If we set $z_i=y_i+|x_i-y_i|z$, then we can repeat the above argument, with $0,y_i,\Gamma_1$ replaced with $y_i,z_i,G_y$, respectively. Notice that even though the $y_i$ are different based points, $G_{y_i}=G_y$ are the same. By iterating for at most $L$ times, one can conclude that $\gamma$ is the identity map on $\R^L$, which is a contradiction.

Now the conclusion follows immediately from the Claim and a covering argument.
\end{proof}
\begin{rem} \label{rem:dia}
For any $(N^k \times \R^{n-k})/\Gamma \in \mathcal M(n,\kappa)$, it follows from the Bonnet-Myers theorem that the diameter of the fiber is at most $\sqrt{2(n-1)} \pi$. Therefore, in practice, we will always assume that $\delta_0 \bar r \ge n^{10}$ so that the base of a fundamental chart has a much larger size than the fiber.
\end{rem}

Now, we recall the distance comparison and existence of the asymptotic Ricci shrinker for a Type-I $\kappa$-solution.

\begin{lem} \label{lem:asym}
Let $(M^n,g(t))_{t \in (-\infty,0]}$ be a nonflat, Type-I, $\kappa$-solution satisfying \eqref{E200}. Then we have the following.
\begin{enumerate}[label=(\roman*)]
\item For any $t_2<t_1 < 0$ and $p,q \in M$, we have
\begin{align*} 
d_{g(t_2)}(p,q)-8(n-1)C_0(\sqrt{|t_2|}-\sqrt{|t_1|})\le d_{g(t_1)}(p,q) \le d_{g(t_2)}(p,q).
\end{align*}
\item For any $p \in M$ and $\tau_i \to \infty$, if we set $g_i(t)=\tau_i^{-1}g(\tau_it)$, then the sequence of Ricci flows $(M,g_i(t),p)_{t \in (-\infty,0)}$ subconverges smoothly to a Ricci flow associated with a Ricci shrinker in $\mathcal M(n,\kappa)$.
\end{enumerate}
\end{lem}
\begin{proof}
Assertion (i) follows directly from the Type-I assumption and \cite[Lemma $8.3$]{Pe1}. For assertion (ii), the existence of the asymptotic Ricci shrinker follows from \cite[Theorem $3.1$]{Na10}. Notice that the limit must be nonflat by \cite[Theorem $4.1$]{CZ11} and hence it belongs to $\mathcal M(n,\kappa)$ by Lemma \ref{lem:shrinker}.
\end{proof}

Now, we classify all compact, Type-I, $\kappa$-solutions.

\begin{thm} \label{thm:cpt}
Let $(M^n,g(t))_{t \in (-\infty,0]}$ be a compact, Type-I, $\kappa$-solution. Then $(M^n,g(t))$ is isometric to a finite quotient of a compact symmetric space $N^n$.
\end{thm}
\begin{proof}
We first prove that the universal cover, denoted by $(\tilde M,\tilde g(t))$, of $(M,g(t))$ is also compact. Indeed, for a fixed point $p$ and any sequence $\tau_i \to \infty$, it follows from Lemma \ref{lem:asym} (ii) that $(M,\tau_i^{-1}g(-\tau_i),p)$ subconverges smoothly to a Ricci shrinker $(M_\infty,g_\infty,p_\infty)\in \mathcal M(n,\kappa)$. Moreover, it follows from the distance estimate in Lemma \ref{lem:asym} (i) that $M_\infty$ is compact and hence is diffeomorphic to $M$. Therefore, the universal cover of $M_\infty$ is compact since the fundamental group of $M_\infty$ is finite by \eqref{E202a}.

We may assume the compact ancient solution $(\tilde M,\tilde g(t))$ is irreducible since otherwise, we can prove the conclusion for each factor. Therefore, we have the following possible cases by Berger's holonomy classification.
\begin{enumerate}[label=(\alph*)]
\item $(\tilde M,\tilde g(t))$ is a compact symmetric space. In this case, there is nothing to prove.

\item $\text{Hol}(\tilde M,\tilde g(t))=\text{SO}(n)$. In this case, one can prove that $(\tilde M,\tilde g(t))$ has strictly PIC$_2$ by using the strong maximum principle developed in \cite{BS09}. For details, refer to \cite[Proposition $6.6$, Case 1]{Bre19}. Therefore, $(\tilde M,\tilde g(t))$ is isometric to the standard sphere $S^n$, see \cite[Corollary $0.4$]{N09}.

\item $\text{Hol}(\tilde M,\tilde g(t))=\text{U}(n/2)$, where $n$ is even. In this case, $(\tilde M,\tilde g(t))$ is a $\kappa$-noncollapsed ancient solution to the K\"ahler Ricci flow. From the weakly PIC$_2$ condition, we derive, in particular, that $(\tilde M,\tilde g(t))$ has nonnegative bisectional curvature. Therefore, $(\tilde M,\tilde g(t))$ is isometric to a compact Hermitian symmetric space from \cite[Theorem $1.1$]{DZ20b}.

\item $\text{Hol}(\tilde M,\tilde g(t)) \ne \text{SO}(n)$ or $\text{U}(n/2)$. In this case, $(\tilde M,\tilde g(t))$ is an Einstein manifold \cite[Chapter $10$]{Be08}. Therefore, $(\tilde M,\tilde g(t))$ is a symmetric space by \cite[Theorem $1$]{Bre10a}.
\end{enumerate}
In sum, $(\tilde M,\tilde g(t))$ is a compact symmetric space, and the proof is complete.
\end{proof}

Now, we turn our attention to noncompact $\kappa$-solutions. First, we prove the following result, which indicates that we only need to consider the simply connected case.

\begin{lem} \label{lem:quo}
Let $(M^n,g(t))$ be a $\kappa$-solution whose universal cover is $N^k \times \R^{n-k}$, where $N$ is a simply connected compact symmetric space. Then the fundamental group of $M$ is finite and
\begin{align*} 
|\pi_1(M) |\le C(n,\kappa),
\end{align*}
where $C(n,\kappa)$ is the same constant as in Lemma \ref{lem:fund1}.
\end{lem}
\begin{proof} 
We assume the fundamental group of $M$ is infinite and derive a contradiction.

From our assumption, $(M^n,g(t))$ is isometric to $(N^k \times \R^{n-k})/\Gamma$, where $\Gamma \subset \text{Iso}(N^k \times \R^{n-k})$ is an infinite discrete subgroup which acts freely on $N^k \times \R^{n-k}$. Moreover, we may assume $L:=n-k \ge 1$, since otherwise, the universal cover is compact.

If we denote the projections of $\Gamma$ to $\text{Iso}(N)$ and $\text{Iso}(\R^L)$ by $\Gamma_0$ and $\Gamma_1$, respectively, then it is clear that $\Gamma_1$ is infinite, since otherwise $(N^k \times \R^{n-k})/\Gamma$ has smaller Hausdorff dimension than $n$. As before, we can regard $(N^k \times \R^{n-k})/\Gamma$ as a singular fibration over the orbifold $\R^L/\Gamma_1$ such that each fiber is a quotient of $N$.

Since $\Gamma_1$ is discrete, after a change of origin in $\R^L$, there exists a normal subgroup $\Gamma_1'$ consists of translating elements, such that the index is finite, see \cite[Theorem $3.2.8$]{WO67}. If we set $V$ to be the subspace generated by elements in $\Gamma_1'$, then $l:=\text{dim}\, V \ge 1$ since $\Gamma_1$ is infinite.

For Ricci flow $\lc((N^k \times \R^{L})/\Gamma,\bar g(t)\rc_{t \in (-\infty,0]}$, it is easy to see that as $t \to -\infty$, the metric on the $\R^L$ direction remains unchanged, but the metric on the fiber direction is rescaled by a factor $|t|$. In addition, the scalar curvature of $\lc((N^k \times \R^{L})/\Gamma,\bar g(t)\rc$ is identically $C|t|^{-1}$. On the other hand, it is clear from our analysis above that for any $x \in (N^k \times \R^{L})/\Gamma$,
\begin{align*} 
\frac{|B_{g(-r^2)}(x,r)|_{g(-r^2)}}{r^{n-l}} \le C
\end{align*} 
for any $r \ge 1$. However, it contradicts our $\kappa$-noncollapsing condition if $r \to \infty$. 

To summarize, we have shown that $\Gamma$ is finite. Now, it is easy to see that any asymptotic Ricci shrinker of $(M^n,g(t))$ is $(N^k \times \R^{n-k})/\Gamma \in \mathcal M(n,\kappa)$. The control of the order of $\Gamma$ follows from Lemma \ref{lem:fund1}.
\end{proof}

\begin{defn} \label{def:close}
Let $(M_1,g_1,x_1)$ and $(M_2,g_2,x_2)$ be two pointed Riemannian manifolds. We say $(M_1,g_1,x_1)$ is $\ep$-close to $(M_2,g_2,x_2)$ if $B_{g_1}(x_1,\ep^{-1})$ and $B_{g_2}(x_2,\ep^{-1})$ are $\ep$-close in $C^{[\ep^{-1}]}$-topology. More precisely, there exist open neighborhoods $U_i$ such that $B_{g_i}(x_i,\ep^{-1}-\ep) \subset U_i \subset \bar U_i \subset B_{g_i}(x_i,\ep^{-1})$ for $i=1,2$. Moreover, there exists a diffeomorphism $\varphi:U_1 \to U_2$ such that $\varphi(x_1)=x_2$ and
\begin{align*} 
\sup_{ U_1}\lc |\varphi^* g_2-g_1|^2+\sum_{i=1}^{[\ep^{-1}]} |\na^i_{g_1} (\varphi^*g_2)|^2\rc \le \ep^2.
\end{align*} 
\end{defn}

\begin{lem} \label{lem:close}
For any $n$ and $\kappa>0$, there exists a small constant $\ep_0=\ep_0(n,\kappa)>0$ satisfying the following property.

Suppose $(N_i^{k_i} \times \R^{n-k_i})/\Gamma_i \in \mathcal M(n,\kappa)$ for $i=1,2$ such that $\lc (N_1^{k_1} \times \R^{n-k_1})/\Gamma_1,x_1 \rc$ is $\ep$-close to $\lc (N_2^{k_2} \times \R^{n-k_2})/\Gamma_2,x_2 \rc$ with $\ep \le \ep_0$. Then $N_1$ is isometric to $N_2$.
\end{lem}

\begin{proof}
From Lemma \ref{lem:finite}, we only need to prove the lemma for two fixed spaces $(N_i^{k_i} \times \R^{n-k_i})/\Gamma_i \in \mathcal M(n,\kappa)$ for $i=1,2$.

From Lemma \ref{lem:chart}, for $i \in \{1,2\}$, there is a fundamental chart $(N_i \times B_{g_E}(0,100 r_i))/\Lambda_i$ in $(N_i^{k_i} \times \R^{n-k_i})/\Gamma_i$ such that $x_i \in (N_i \times B_{g_E}(0, r_i))/\Lambda_i$. As discussed in Remark \ref{rem:dia}, we assume $r_i \in [n^{10}, \delta_0^{-1} n^{10}]$ and, without loss of generality, $r_1 \le r_2$. Moreover, we denote the projection map from $N_i \times B_{g_E}(0, 100 r_i)$ to $(N_i \times B_{g_E}(0, 100 r_i))/\Lambda_i$ by $\pi_i$.

By our assumption, there exists a map $\varphi$ from a neighborhood of $x_1$ to that of $x_2$ which is almost isometric. We denote a lift of $x_i$ to $N_i \times B_{g_E}(0,100 r_i)$ by $\tilde x_i$. Since $N_1 \times B_{g_E}(0,r_1)$ is simply connected, we can lift $\varphi$ to a smooth map
\begin{align*} 
\tilde \varphi: N_1 \times B_{g_E}(0, r_1) \to N_2 \times B_{g_E}(0, 100 r_2),
\end{align*} 
such that $\tilde \varphi(\tilde x_1)=\tilde x_2$. Now we claim that $\tilde \varphi$ is a smooth embedding. It is clear that $\tilde \varphi$ is a local diffeomorphism since $\varphi$ is a diffeomorphism. Therefore, we only need to prove that $\tilde \varphi$ is injective. Suppose there are $a,b \in N_1 \times B_{g_E}(0, r_1)$ with $a \ne b$ such that $\tilde \varphi(a)=\tilde \varphi(b)=c$. From our construction, we set $z:=\pi_1(a)=\pi_1(b)$ and $w:=\varphi(z)=\pi_2(c)$. We connect $a$ with $b$ by a geodesic segment $\tilde \gamma_1$ contained in $N_1 \times B_{g_E}(0, r_1)$ so that the loop $\gamma_1:=\pi_1 \circ \tilde \gamma_1$ represents a nontrivial element in $\Lambda_1$. In addition, if we set $\tilde \gamma_2:=\tilde \varphi\circ \tilde \gamma_1$ and $\gamma_2:=\pi_2 \circ \tilde \gamma_2$, then $\gamma_2$ is a trivial element in $\Lambda_2$. Therefore, there exists a homotopy $F$ of $\gamma_2$ to $w$ such that the image of $F$ is contained in $B(x_2,10r_1)$. Since $\varphi$ is a diffeomorphism that is almost isometric, we can pull back the homotopy $F$, and it implies that $\gamma_1$ is homotopic to $z$ in $(N_1 \times B_{g_E}(0, 100 r_1))/\Lambda_1$, which is a contradiction.

To summarize, $\tilde \varphi$ is an embedding that is almost isometric. In particular, $\tilde \varphi$ almost preserves the parallel directions $\R^{n-k_1}$. If $\ep$ is sufficiently small, it is easy to see $k_1=k_2$ and $\tilde \varphi$ almost preserves the vertical direction as well. Therefore, if we denote the copies of $N_1$ and $N_2$ through $\tilde x_1$ and $\tilde x_2$ by $N_1'$ and $N_2'$ respectively, then the normal projection of $\tilde \varphi(N_1')$ to $N'_2$ is almost an isometry, see, e.g., \cite[Appendix 2]{CFG92}. Therefore, $N_1$ is isometric to $N_2$ if $\ep$ is sufficiently small.
\end{proof}

\begin{rem}\label{rem:inj}
The proof of Lemma \ref{lem:close} indicates that if $\ep$ is sufficiently small, then $\varphi$ injects $\Lambda_1'$ into $\Lambda_2$, where $\Lambda_1'$ is the group of isometries obtained by restriction of $\Lambda_1$ on $N_1 \times B_{g_E}(0, r_1)$.
\end{rem}

Now we prove the main result of this section.

\begin{thm} \label{thm:noncpt}
Let $(M^n,g(t))_{t \in (-\infty,0]}$ be a simply connected, noncompact, Type-I, $\kappa$-solution. Then $(M^n,g(t))$ is isometric to $N^k \times \R^{n-k}$, where $N$ is a simply connected compact symmetric space.
\end{thm}
\begin{proof}
We prove the theorem by induction. If $n=2$, the conclusion obviously holds, since all $2$-dimensional $\kappa$-solutions are $S^2,\RP^2$ and $\R^2$, proved by Perelman \cite{Pe1}. Now we assume the conclusion holds for any dimension smaller than $n$. 

Since $(M^n,g(t))$ has weakly PIC$_2$, it follows from \cite[Theorem $1.2$]{CW15} that $(M^n,g(t))$ is isometric to $\Sigma \times F$, where $\Sigma$ is the $k$-dimensional soul and $F$ is diffeomorphic to $\R^{n-k}$. If $k \ge 1$, it is clear that $(M^n,g(t))$ splits as Type-I $\kappa$-solutions on $\Sigma$ and $F$. Therefore, the conclusion follows from our inductive assumption.

We assume $k=0$ and hence $M$ is diffeomorphic to $\R^n$. If $(M,g(t))$ is flat, then the proof is complete. For the rest of the proof, we assume that $(M,g(t))$ is nonflat and derive a contradiction. 

Let $p$ be a fixed point in $M$ and $\delta=\delta(\ep)$ a positive function such that $\lim_{\ep \to 0} \delta (\ep)=0$. In the following proof, $\delta$ may differ line by line.

\textbf{Claim 1}: For any $\ep>0$, there exists a $T=T(\ep)>0$ such that for any $\tau \ge T$, $(M,\tau^{-1}g(\tau t),p)$ is $\ep$-close to a space $(M_{\infty},g_{\infty}(t),p_{\infty})$ for $t \in [-10,-1/10]$, where $(M_{\infty},g_{\infty}(-1))\in \mathcal M(n,\kappa)$.

Indeed, suppose the claim does not hold for some $\ep$ and a sequence $\tau_i \to \infty$, then by Lemma \ref{lem:asym}, $(M,\tau_i^{-1}g(\tau_i t),p)$ subconverges smoothly to a Ricci flow associated with a Ricci shrinker in $\mathcal M(n,\kappa)$, which is a contradiction.

Next we define a sequence $\tau_i:=4^i$ and set $g_i(t)=\tau_i^{-1}g(\tau_i t)$ and $g_i=g_i(-1)$. Then it is clear that
\begin{align}
g_{i-1}=\tau_{i-1}^{-1} g(-\tau_{i-1})=4 \tau_{i}^{-1}g(-\tau_{i}/4)=4 g_{i}(-1/4). \label{E203a}
\end{align} 

We set $I$ to be the smallest integer such that $\tau_{I-1} \ge T(\ep)$. For any $i \ge I$, it follows from Claim 1 that $(M,g_i,p)$ is $\ep$-close to a Ricci shrinker $ \lc N_i^{k_i} \times \R^{n-k_i})/\Gamma_i,x_i \rc \in \mathcal M(n,\kappa)$. By Definition \ref{def:close}, there exists a diffeomorphism $\varphi_i$ from an open neighborhood containing $B_{g_i}(p,4\ep^{-1}/5)$ onto an open neighborhood containing $B(x_i,4\ep^{-1}/5) \subset (N_i^{k_i} \times \R^{n-k_i})/\Gamma_i$ such that $\varphi_i(p)=x_i$. As discussed earlier, we regard $(N_i^{k_i} \times \R^{n-k_i})/\Gamma_i$ as a singular fibration $(N_i^{k_i} \times \R^{n-{k_i}})/\Gamma_i \overset{\pi_i}{\longrightarrow} \R^{n-k_i}/\Gamma_{1,i}$. If we denote the induced distance function on $\R^{n-k_i}/\Gamma_{1,i}$ by $d_i$, then we define
\begin{align*} 
W_i:=&B_{d_i}(\pi_i(x_i), \ep^{-1}/4), \quad W_i':= B_{d_i}(\pi_i(x_i), \ep^{-1}/2), \\
V_i:=&\varphi_i^{-1} \lc \pi_i^{-1} (W_i) \rc, \quad \quad \, \, U_i:=\varphi_i^{-1} \lc \pi_i^{-1}(W'_i) \rc.
\end{align*} 
In particular, we obtain a singular fibration structure on $U_i$. 

\textbf{Claim 2}: $N_i$ is isometric to $N_{i-1}$ for all $i \ge I$, if $\ep$ is sufficiently small.

From the fibration structure, there exist an open set $\widetilde W_{i} \subset \R^{n-k_i}$, two subgroups $G_i \subset \Gamma_{1,i}$ and $\Lambda_{i} \subset \Gamma_{i}$, and a map $\psi_i$ such that $\widetilde W_i/G_i$ is an orbifold chart and $\psi_i: (N_{i}\times \widetilde W_i)/\Lambda_{i} \to \pi_{i}^{-1}(W_{i})$ is an isometry. Since $G_i$ is finite, there exists a point $y_i \in \widetilde W_i$ fixed by $G_i$. For simplicity, we assume $y_i=0$ and define 
\begin{align*} 
\widetilde W'_i=\left \{2x \,\mid x\in \widetilde W_i \right \}.
\end{align*} 
Therefore, we have the diffeomorphism
\begin{align*}
(N_i \times \widetilde W'_i) /\Lambda_i \overset{\iota_i}{\longrightarrow}(N_i \times \widetilde W_i) /\Lambda_i
\end{align*} 
such that $\iota_i$ multiplies the horizontal coordinates by $1/2$. From Claim 1 and \eqref{E203a}, we immediately conclude that $(V_i,g_{i-1},p)$ is $\delta$-close to $\lc (N_i \times \widetilde W'_i) /\Lambda_i,z_i \rc$, where $z_i:=(\psi_i\circ \iota_i)^{-1}(x_i)$. Here, we have used the fact that the vertical direction in $4 g_{i}(-1/4)$ is almost unchanged.

Since $(M,g_{i-1},p)$ is $\ep$-close to $\lc (N_{i-1} \times \R^{n-k_{i-1}}) /\Lambda_{i-1}, x_{i-1} \rc$, it is clear from Lemma \ref{lem:close} that $N_i$ is isometric to $N_{i-1}$ for any $i \ge I$, if $\ep$ is sufficiently small. From now on, we denote all $N_i$ by a common $N=N^k$.

\textbf{Claim 3}: $V_i$ is simply connected for any $i \ge I$, if $\ep$ is sufficiently small.

To prove Claim 3, we only need to show that each fiber in $V_i$ is simply connected. Indeed, if all fibers are simply connected, $V_i$ is a regular fibration over $B_{g_E}(\pi_i(x_i),\ep^{-1}/4)$ with fiber $N$. Since $B_{g_E}(\pi_i(x_i),\ep^{-1}/4)$ is contractible, we conclude that $V_i$ is diffeomorphic to $N \times B_{g_E}(0,1)$.

Suppose the conclusion does not hold for some $j-1 \ge I$, then there exists a point $x \in V_{j-1}$ such that the fiber $F_{j-1}$ containing $x_{j-1}$ is not simply connected. Therefore, there exists a loop $\gamma_{j-1}$ based at $x_{j-1}$, which is contained in $F_{j-1}$ and represents a nontrivial element in $\pi_1(F_{j-1})$. In particular, it follows from the structure of the fibration that $\gamma_{j-1}$ also represents a nontrivial element in $\pi_1(U_{j-1})$.

On the one hand, since $M$ is simply connected, there exists a homotopy $\Psi(u,v) :[0,1] \times [0,1] \to M$ such that $\Psi(0,v)=\gamma_{j-1}(v)$ and $\Psi(1,v)=\Psi(u,0)=\Psi(u,1)=x_{j-1}$ for any $(u,v) \in [0,1] \times [0,1]$. Since the image of $\Psi$ is compact, we assume that
\begin{align*} 
\text{Image}\, (\Psi) \subset B_{g_{j-1}}(p,K) 
\end{align*} 
for some large constant $K>0$. 

From the above arguments, we know that $(M,g_{j-1},p)$ is $\delta$-close to $\lc (N \times \widetilde W'_j) /\Lambda_j,z_j \rc$. As in the proof of Lemma \ref{lem:close} (see Remark \ref{rem:inj}), we conclude that there exists a loop $\gamma_j$ based at a point $x_j \in V_j$ such that $\gamma_j$ is contained in the fiber $F_j$ through $x_j$ and $\gamma_j$ represents a nontrivial element in $\pi_1(F_j)$. Moreover, $\gamma_{j-1}$ is homotopic to $\eta_j *\gamma_j $, where $\eta_j$ is a geodesic segment with respect to $g_{j-1}$ connecting $x_{j-1}$ with $x_j$ and $*$ is the joining of two curves. Besides, the image of the homotopy $\Psi_j$ is contained in $B_{g_{j-1}}(x_{j-1},L)$ for some constant $L$ depending only on $n$ and $\kappa$. By iteration, for any $i \ge j$, we can successively define a loop $\gamma_{i}$ which is based at $x_i$ and contained in the fiber $F_i$ through $x_i$ and a homotopy $\Psi_i$ between $\gamma_{i-1}$ and $\eta_i * \gamma_i$, where $\eta_i$ is a geodesic segment with respect to $g_{i-1}$ connecting $x_{i-1}$ with $x_i$. Moreover, the image of $\Psi_i$ is contained in $B_{g_{i-1}}(x_{i-1},L)$ and $\gamma_i$ represents a nontrivial element in $\pi_1(F_i)$. Notice that $\gamma_i$ also represents a nontrivial element in $\pi_1(U_i)$.

Now we define a sequence $\{a_i\}_{i \ge j-1}$ such that $a_{j-1}=0$ and $a_{i+1}=4(n-1)C_0+\frac{a_i+L}{2}$. Then it follows immediately from Lemma \ref{lem:asym} (i) that for any $i \ge j$,
\begin{align*} 
\text{Image}\, \lc \Psi_{j}* \cdots * \Psi_{i} \rc \subset B_{g_{i}}(x_{j-1},a_i),
\end{align*} 
where we use $*$ to denote the joining of homotopies. Clearly, the sequence $a_i \le L+8(n-1)C_0$ and hence $\gamma_{j-1}$ is homotopic to $\gamma_i$ in $U_i$ for any $i \ge j$, if $\ep$ is sufficiently small. In particular, $\gamma_{j-1}$ represents a nontrivial element in $\pi_1(U_i)$.

On the other hand, it follows from Lemma \ref{lem:asym} (i) that 
\begin{align*} 
B_{g_{j-1}}(p,K) \subset B_{g_i}(p,\ep^{-1}/8)
\end{align*} 
if $i$ is sufficiently large. Therefore, the image of the homotopy $\Psi$ is contained in $U_i$ and hence the $\gamma_{j-1}$ is trivial in $\pi_1(U_i)$. From this, we derive a contradiction. 

In sum, we have shown that the fibration structure $V_i$ is trivial. In particular, for any $i \ge I$, there exists a trivial fibration
\begin{align*} 
N \hookrightarrow V_i \overset{f_i}{\longrightarrow} B_{g_E}(0,\ep^{-1}/4).
\end{align*} 
Moreover, $(V_i,g_i,p)$ is $\delta$-close to $N^k \times \R^{n-k}$. From Lemma \ref{lem:asym} (i), it is clear that $V_I \subset V_{I+1} \subset \cdots$ is an exhaustion of $M$. Therefore, it follows from a standard argument, see, for example, \cite[Lemma $1.4$]{CG}, that we can modify and glue all those fibrations so that we obtain a global fibration
\begin{align*} 
N \hookrightarrow M \overset{f}{\longrightarrow} Y.
\end{align*} 
where $Y$ is an open manifold. 

Since $M$ is diffeomorphic to $\R^n$, we obtain a contradiction from \cite{BS50}, which states that for any fibration on $\R^n$ with compact fibers, the fiber must be a single point.
\end{proof}

\emph{Proof of Theorem \ref{thm:main2}}: It follows immediately from Theorem \ref{thm:cpt}, Lemma \ref{lem:quo} and Theorem \ref{thm:noncpt}.

From \cite[Lemma 4.2]{BCW19}(see also \cite[Proposition 6.2]{LN20}), we know that any complete ancient solution to the Ricci flow with weakly PIC$_1$ automatically has weakly PIC$_2$. Therefore, the following corollary is immediate from Theorem \ref{thm:main2}.

\begin{cor}
Let $(M^n,g(t))_{t \in (-\infty,0]}$ be a $\kappa$-noncollapsed, Type-I, complete ancient solution to the Ricci flow with weakly \emph{PIC$_1$}. Then it is isometric to a finite quotient of $N^k \times \R^{n-k}$, where $N$ is a simply connected compact symmetric space.
\end{cor}

\section{Proof of the main theorem}

We recall the following definitions in the K\"ahler Ricci flow. In the following, we always denote the complex dimension by $m$ and the real dimension by $n=2m$. 

\begin{defn}\label{def:flow 2} Suppose that $(M^m,g(t))_{t \in (-\infty,0]}$ be a complete ancient solution to the K\"ahler Ricci flow $\partial_t g=-2 \Rc$ on a complex manifold $M$.
\begin{enumerate}[label=(\roman*)]

\item $(M,g(t))_{t \in (-\infty,0]}$ has nonnegative bisectional curvature \emph{($BK \ge 0$ for short)} if for any $(x,t) \in M \times (-\infty,0]$ and any $Z,W \in T^{1,0}_x M$, we have at $(x,t)$,
\begin{align*}
\emph{Rm}(Z,\bar W,\bar Z, W) \ge 0.
\end{align*}

\item \emph{($\kappa$-solution to the K\"ahler Ricci flow)}
$(M,g(t))_{t \in (-\infty,0]}$ is called a $\kappa$-solution if it has $BK \ge 0$, uniformly bounded curvature and is $\kappa$-noncollapsed.
\end{enumerate}
\end{defn}

Notice that the K\"ahler Ricci flow defined here agrees with the real Ricci flow, but differs by a constant $2$ from the convention $\partial_t g_{i \bar j}=-R_{i \bar j}$. Moreover, the curvature operator here is different by a sign from that in most articles in K\"ahler geometry. The concepts of scalar curvature, Laplacian, etc., always refer to the concepts in real Riemannian geometry.

\subsection*{3.1\quad$BK \ge 0$ implies weakly PIC$_2$}
\addtocontents{toc}{\protect\setcounter{tocdepth}{2}}

First, we prove that any $\kappa$-solution to the K\"ahler Ricci flow is of Type-I. The proof is similar to \cite[Lemma $6.2$]{CL20}.

\begin{lem}\label{lem:type1}
Let $(M^m,g(t))_{t \in (-\infty,0]}$ be a $\kappa$-solution to the K\"ahler Ricci flow. Then there exists a constant $C_0>0$ such that
\begin{align*}
|\emph{Rm}|(x,t) \le \frac{C_0}{1+|t|}.
\end{align*}
for any $(x,t) \in M \times (-\infty,0]$.
\end{lem}

\begin{proof} 
Suppose the $(M^m,g(t))_{t \in (-\infty,0]}$ is of Type-II. We take any $T_i \to -\infty$ and $\delta_i \to 0^+$ and choose $(x_i,t_i) \in M \times [T_i,0]$ so that 
\begin{align*}
|t_i|(t_i-T_i)R(x_i,t_i)\ge (1-\delta_i)\sup_{M \times [T_i,0]} |t|(t-T_i)R(x,t)
\end{align*}
Notice that the choice is possible since the curvature is assumed to be uniformly bounded on $M \times (-\infty,0]$. For the rescaled metric $g_i(t):=Q_i g(t_i+Q_i^{-1}t)$ where $Q_i:=R(x_i,t_i)$, it follows by direct computation-see \cite[Proposition $8.20$]{CLN06}-that $(M,g_i(t),x_i)$ converges smoothly to $(M_\infty,g_\infty(t),x_\infty)_{t \in (-\infty,\infty)}$ which is an eternal $\kappa$-solution to the K\"ahler Ricci flow. Moreover, $R_{g_{\infty}}(x,t)\le 1$ on $M_{\infty} \times (-\infty,\infty)$ and $R_{g_{\infty}}(x_\infty,0)= 1$ by the construction. After taking the universal cover of $M_\infty$ and using Cao's dimension reduction argument \cite[Theorem $2.1$]{Cao04}, we may assume that $M_\infty$ is simply connected and has positive Ricci curvature. Therefore, it follows from \cite[Theorem $1.3$]{Cao97} that $(M_\infty,g_\infty)$ is a nonflat, $\kappa$-noncollapsed K\"ahler Ricci steady soliton with nonnegative bisectional curvature. However, such a K\"ahler steady soliton is excluded by \cite[Theorem $1.2$]{DZ20a}.
\end{proof}

Next, we obtain the following curvature improvement for $\kappa$-solutions. Notice that a stronger curvature improvement was proved for K\"ahler surfaces in \cite[Lemma $4.6$]{CL20}.

\begin{thm} \label{thm:pic2}
Let $(M^m,g(t))_{t \in (-\infty,0]}$ be a complete ancient solution to the K\"ahler Ricci flow with $BK \ge 0$. Then it has weakly \emph{PIC}$_2$.
\end{thm}

\begin{proof}
By taking the universal cover, we may assume $M$ is simply connected. We prove the conclusion by induction for the dimension $m$. If $m=1$, the conclusion trivially holds. We assume the conclusion is true for any dimension smaller than $m$.

If the Ricci curvature vanishes at some spacetime point, then it follows from \cite[Theorem $2.1$]{Cao04} that $(M^m,g(t))$ is isometrically biholomorphic to $\Sigma \times \C^{k}$ for some $k \ge 1$ such that $\Sigma$ has positive Ricci curvature. Notice that the boundedness of the curvature is not necessary for the proof of \cite[Theorem $2.1$]{Cao04}. Therefore, $(M^m,g(t))$ splits as two ancient solutions to the K\"ahler Ricci flow with $BK \ge 0$. Then the conclusion follows from the inductive assumption.

For this reason, we may assume the Ricci curvature is positive on the spacetime. From the strong maximum principle, see \cite[Proposition $1.1$]{Mok88}, we conclude that $(M^m,g(t))$ has positive holomorphic sectional curvature. If $(M^m,g(t))$ is symmetric, then the conclusion obviously holds. Otherwise, the holonomy group of $M$ is $\text{U}(m)$ and the proof of \cite[Theorem $1.2$]{Gu09}, which is based on the strong maximum principle in \cite{BS09}, shows that $(M^m,g(t))$ has $BK>0$.

If we regard the curvature operator $\Rm$ as a symmetric operator on $2$-forms $\Lambda^2=\Lambda^2 M$, then $BK > 0$ is equivalent to 
\begin{align}
\Rm(v,\bar v) > 0 \label{E301a}
\end{align}
for any $v=Z \wedge \bar W \in \Lambda^{1,1}$ with $v \ne 0$, where $Z,W \in \Lambda^{1,0}$. 

Next, define a function $\lambda$ on $M \times (-\infty,0]$ as
\begin{align*}
\lambda(x,t)=\inf_{v \in S(x,t)}{\Rm(x,t)(v,\bar v)}, 
\end{align*} 
where
\begin{align}
S(x,t)=\left\{ v \, \mid \, v=Z \wedge \bar W,\,\, Z,W \in \Lambda^1_x \, \text{with}\, \,|P(Z)|_t= |P(W)|_t=1 \right\}. \label{E301bx}
\end{align} 
Here $P: \Lambda^1 \to \Lambda^{0,1}$ is the projection and the norm $|\cdot|_t$ comes from the Hermitian inner product on $ \Lambda^1$ induced by $g(t)$.

At a spacetime point $(x_0,t_0) \in M \times (-\infty,0]$ such that $\lambda(x_0,t_0) \le 0$, we claim that a minimizer $v$ for $\lambda$ exists. To prove this, we set $v_i=Z_i \wedge \bar W_i$ to be a minimizing sequence such that
\begin{align*}
\lim_{i \to \infty} \Rm(x_0,t_0) (v_i,\bar v_i) =\lambda(x_0,t_0).
\end{align*} 
Now we set $Z_i=a_i+\bar b_i$ and $W_i=c_i+\bar d_i$, where $a_i,b_i,c_i,d_i \in \Lambda_{x_0}^{1,0}$. From our definition, we have $|b_i|=|d_i|=1$. By taking a subsequence, we may assume
\begin{align*}
\lim_{i \to \infty} b_i=b \quad \text{and} \quad \lim_{i \to \infty} d_i=d 
\end{align*}
where $b,d \in \Lambda_{x_0}^{1,0}$ and $|b|=|d|=1$. From a direct calculation,
\begin{align}
\Rm(Z_i,\bar W_i,\bar Z_i,W_i)= \Rm(a_i,\bar c_i, \bar a_i, c_i)+2 \mathfrak{Re}(\Rm(a_i,\bar c_i,b_i,\bar d_i))+\Rm(b_i,\bar d_i,\bar b_i,d_i). \label{E301ac}
\end{align}
By taking a subsequence if necessary, there are three possible cases.

\emph{Case} 1: $\lim_{i \to \infty} |a_i| |c_i|=0$.
In this case, it is easy to see from \eqref{E301a} and \eqref{E301ac} that
\begin{align*}
\lambda=\lim_{i \to \infty} \Rm(Z_i,\bar W_i,\bar Z_i,W_i) =\Rm(b,\bar d,\bar b,d) >0,
\end{align*} 
which contradicts our assumption that $\lambda(x_0,t_0) \le 0$.

\emph{Case} 2: $\lim_{i \to \infty} |a_i| |c_i|=+\infty$. In this case, we rewrite the right-hand side of \eqref{E301ac} as
\begin{align*}
\Rm(Z_i,\bar W_i,\bar Z_i,W_i)= (|a_i||c_i|)^2\Rm(a'_i,\bar c'_i, \bar a'_i, c'_i)+2|a_i||c_i| \mathfrak{Re}(\Rm(a'_i,\bar c'_i,b_i,\bar d_i))+\Rm(b_i,\bar d_i,\bar b_i,d_i),
\end{align*}
where $a_i'=a_i/|a_i|$ and $c_i'=c_i/|c_i|$. Therefore, it follows immediately from \eqref{E301a} that
\begin{align*}
\lambda=\lim_{i \to \infty} \Rm(Z_i,\bar W_i,\bar Z_i,W_i) =+\infty,
\end{align*}
which is impossible.

\emph{Case} 3: $\lim_{i \to \infty} |a_i| |c_i|=\tau \in (0,\infty)$. In this case, we assume 
\begin{align*}
\lim_{i \to \infty} \frac{a_i}{|a_i|}=a \quad \text{and} \quad \lim_{i \to \infty} c_i|a_i|=c
\end{align*}
where $a,c \in \Lambda_{x_0}^{1,0}$ with $|a|=1$ and $|c|=\tau$. Then it is clear from \eqref{E301ac} that
\begin{align*}
\lambda=&\lim_{i \to \infty} \Rm(Z_i,\bar W_i,\bar Z_i,W_i) \\
=&\Rm(a,\bar c, \bar a, c)+2 \mathfrak{Re}(\Rm(a,\bar c,b,\bar d))+\Rm(b,\bar d,\bar b,d) = \Rm(Z,\bar W,\bar Z,W)
\end{align*} 
where $Z:=a+\bar b$ and $W=c+\bar d$. In other words, $v:=Z \wedge \bar W \in S(x_0,t_0)$ is a minimizer of $\lambda$. In particular, $\lambda >-\infty$.

Now we define $u=Z \wedge d$, $w=Z \wedge \bar c$ and $v(s)=u+sw$ for any $s \in \R$. It is clear that $v(s) \in S$ and $v=u+w=v(1)$. Since $\Rm(v(s),\bar v(s))$ attains the minimum at $s=1$, by taking the derivative, we have
\begin{align*}
\mathfrak{Re}(\Rm(u,\bar w))+\Rm(w,\bar w)=0 
\end{align*}
and hence
\begin{align}
\lambda=\Rm(u,\bar u)+2\mathfrak{Re}( \Rm(u,\bar w))+\Rm(w,\bar w)= \mathfrak{Re}(\Rm(v,\bar u)). \label{E301ca}
\end{align}
From \eqref{E301ca}, we obtain
\begin{align}
|\lambda| \le |\Rm(v,\bar u)|=|\Rm(Z,\bar W,\bar Z,\bar d)|=|\Rm(Z,\bar W, b,\bar d)| \le |\Rm(v)|, \label{E301cb}
\end{align}
where we have used the fact that $|b|=|d|=1$. 

Next, by applying Uhlenbeck's trick, the Riemannian curvature under the Ricci flow is deformed by
\begin{align}
\tl \Rm=\Rm^2+\Rm^{\#}, \label{E301cbx}
\end{align}
see \cite[Lemma 2.58]{CLN06}. For each point $(x_0,t_0)$ such that $\lambda(x_0,t_0) \le 0$, we choose a minimizer $v(x_0,t_0) \in S(x_0,t_0)$ in the definition of $\lambda$. After one applies Uhlenbeck's trick by using a 1-parameter family of bundle isomorphism, the pullback of $g(t)$ is independent of $t$. Therefore, $v(x_0,t_0) \in S(x_0,t)$ from the definition \eqref{E301bx}. Extending $v$ by parallel transport, one obtains a $2$-form $v$ locally defined on a spacetime neighborhood of $(x_0,t_0)$ such that $v(x,t) \in S(x,t)$.

From the evolution equation of $\Rm$ \eqref{E301cbx} and \eqref{E301cb}, we conclude that in the barrier sense,
\begin{align}
\tl \lambda \ge |\Rm(v)|^2+\Rm^{\#}(v,\bar v) \ge \lambda^2+\Rm^{\#}(v,\bar v)
\label{E301cd}
\end{align}
at $(x_0,t_0)$. Now it follows from the same proof of \cite[Theorem $1$]{Wil13} that $\Rm^{\#}(v,\bar v) \ge 0$. Indeed, if we consider the Lie algebra $\mathfrak{g}=\mathfrak{u}(m) \otimes _{\R} \C$, then the set $S(x_0,t_0)$ in \eqref{E301bx} is preserved by the adjoint representation of the Lie group with Lie algebra $\mathfrak{g}$. Therefore, the same proof of \cite[Theorem $1$]{Wil13} yields the claim. 

Therefore, \eqref{E301cd} implies that
\begin{align*}
\tl \lambda^- \ge (\lambda^-)^2
\end{align*}
in the barrier sense on $M \times (-\infty,0]$, where $\lambda^-:=\min\{\lambda,0\}$. Now, it follows from \cite[Corollary $2.4$]{CL20} that $\lambda \ge 0$ on $M \times (-\infty,0]$.

From the definition of $\lambda$, we conclude that $(M^m,g(t))$ has weakly PIC$_2$. Indeed, from Definition \ref{def:flow1} (iii) we only need to prove 
\begin{align}
\Rm(Z,\bar W,\bar Z,W) \ge 0 \label{E301cc}
\end{align}
for any $Z,W \in \Lambda^1$. We set $Z=a+\bar b$ and $W=c+\bar d$ as before, where $a,b,c,d \in \Lambda^{1,0}$. If $|b||d| \ne 0$, then
\begin{align*}
\Rm(Z,\bar W,\bar Z,W) =(|b||d|)^2\Rm(Z',\bar W',\bar Z',W') \ge 0
\end{align*}
where $Z'=Z/|b|$ and $W'=W/|d|$. If $b=0$, then
\begin{align*}
\Rm(Z,\bar W,\bar Z,W) =\Rm(a,\bar c,\bar a,c) \ge 0
\end{align*}
since $BK \ge 0$. Similarly, \eqref{E301cc} also holds for the case $d=0$. 

In sum, the proof is complete.
\end{proof}

Now, we can classify all $\kappa$-solutions.

With the help of Theorem \ref{thm:main2}, Lemma \ref{lem:type1} and Theorem \ref{thm:pic2}, the classification of all $\kappa$-solutions is immediate.

\begin{thm} \label{thm:ka}
Let $(M^m,g(t))_{t \in (-\infty,0]}$ be a $\kappa$-solution to the K\"ahler Ricci flow. Then $(M^m,g(t))$ is isometrically biholomorphic to $N^k \times \C^{m-k}$, where $N$ is a compact Hermitian symmetric space.
\end{thm}

\begin{proof}
By Lemma \ref{lem:type1}, Theorem \ref{thm:pic2} and Theorem \ref{thm:main2}, $(M^m,g(t))$ is isometrically biholomorphic to $(N^k \times \C^{m-k})/\Gamma$, where $N$ is a compact Hermitian symmetric space and $\Gamma$, acting freely on $N^k \times \C^{m-k}$, is a finite subgroup of the holomorphic isometry group of $N^k \times \C^{m-k}$.

To complete the proof, we need to show that $\Gamma$ is trivial. For any $\sigma \in \Gamma$, write $\sigma=(\sigma_1, \sigma_2) \in \mathrm{Iso}(N^k) \times \mathrm{Iso}(\C^{m-k})$, using the uniqueness of de Rham's decomposition theorem (see \cite[Chapter IX, Theorem $8.1$]{KoNo63}), where $\mathrm{Iso}(\cdot)$ denotes the holomorphic isometry group. Since $\sigma$ has finite order, the center of mass construction ensures the existence of $q \in \C^{m-k}$ fixed by $\sigma_2$. Consequently, $\sigma_1$ acts freely on $N^k$.  

Next, we decompose $N^k=N_1 \times \cdots \times N_l$, where each $N_i$ is a compact Hermitian symmetric space equipped with a K\"ahler-Einstein metric with Einstein constant $\lambda_i>0$, and $\lambda_i \ne \lambda_j$ for any $1 \le i <j \le l$. This decomposition is valid because $N^k$ can first be expressed as the product of irreducible compact Hermitian symmetric spaces, each admitting a positive K\"ahler-Einstein metric. Then, we group them by their distinct Einstein constants.

By the uniqueness of de Rham's decomposition theorem, we write $\sigma_1=(\sigma^1, \cdots, \sigma^l) \in \mathrm{Iso}(N_1) \times \cdots \times \mathrm{Iso}(N_l)$. If $\sigma_1$ is nontrivial, assume without loss of generality that a nontrivial component $\sigma^1$ acts freely on $N_1$. This would imply that $N_1/\la \sigma^1 \ra$ admits a positive K\"ahler-Einstein metric. However, this contradicts Kobayashi's theorem \cite{Ko61}, which asserts that any positive K\"ahler-Einstein manifold must be simply connected.

Thus, $\sigma_1$ must be trivial, and consequently, $\sigma$ is trivial as well. This completes the proof.
\end{proof}

\subsection*{3.2\quad$BK \ge 0$ implies bounded curvature}

In this subsection, we aim to prove that any complete, $\kappa$-noncollapsed, ancient solution to the K\"ahler Ricci flow with $BK \ge 0$ must have uniformly bounded curvature and hence is a $\kappa$-solution.

\begin{defn} 
The moduli space $\mathcal N=\mathcal N(m,\kappa)$ consists of all noncompact, nonflat, $\kappa$-noncollapsed, K\"ahler manifolds in the form of $N^k \times \C^{m-k}$, where $N$ is a compact Hermitian symmetric space normalized such that its scalar curvature is identically $1$.
\end{defn}

It follows from de Rham's decomposition \cite[Chapter XI, Theorem $8.1$]{KoNo63} that for any $(N^k \times \C^{m-k} ,\bar g) \in \mathcal N(m,\kappa)$, we have 
\begin{align}
N=N^k=N_0 \times N_1 \times \cdots \times N_s,
\label{E302ab}
\end{align}
where each $N_i$ is an irreducible compact Hermitian symmetric space. From Lemma \ref{lem:finite}, there are finitely many holomorphic isometry classes in $\mathcal N(m,\kappa)$, up to scaling on each factor $N_i$. From the decomposition \eqref{E302ab}, we assume $N_0$ has the maximal scalar curvature. In particular, 
\begin{align} 
1\ge R_{N_0} \ge \frac{1}{m} \label{E303bb}
\end{align} 
and hence the diameter of $N_0$ is bounded above by $ m \pi$. For each $q \in N^k \times \C^{m-k}$, there exists a unique copy $N_0$, denoted by $N_{0,q}$, containing $q$. 

We set the projection from $N^k \times \C^{m-k}$ to $\C^{m-k}$ by $\pi$, fix a base point $\bar q \in N^k \times \C^{m-k}$ with $\bar x=\pi(\bar q)$, and consider the following open sets:
\begin{align} 
U':=\pi^{-1} \lc B_E(\bar x,\ep^{-1}/2) \rc, \quad V':=\pi^{-1} \lc B_E(\bar x,\ep^{-1}/4) \rc \quad \text{and} \quad W':=\pi^{-1} \lc B_E(\bar x,\ep^{-1}/10) \rc
\label{E303ba}
\end{align} 
for a small constant $\ep$ to be determined later. Here, $B_E$ denotes the ball in $\C^{m-k}$.

For any $q_1,q_2 \in V'$, we define $N_{0,q_1} \circeq N_{0,q_2}$ if either $N_{0,q_1}= N_{0,q_2}$ or there exists a smooth manifold $T$ embedded in $U'$ such that $T$ is diffeomorphic to $N_0 \times [0,1]$ and $\partial T=N_{0,q_1} \cup N_{0,q_2}$. Moreover, we say $N_{0,q_1} \equiv N_{0,q_2}$ if there exist $\{x_1,x_2,\cdots, x_s\} \subset V'$ such that $q_1=x_1$, $q_2=x_s$ and $N_{0,x_i} \circeq N_{0,x_{i+1}}$ for any $1\le i \le s-1$. It is clear from the definition that $\equiv$ is an equivalence relation.

\begin{lem}\label{lem:equi}
With the above assumptions, $N_{0,q_1} \equiv N_{0,q_2}$ for any $q_1,q_2 \in V'$.
\end{lem}
\begin{proof}
For any $q_1,q_2 \in V'$, we first assume that $q_1$ and $q_2$ belong to the same fiber, meaning $\pi(q_1)=\pi(q_2)=x$ for some $x \in B_E(\bar x,\ep^{-1}/4)$. Write $q_i=(a_i,b_i,x) \in N_0 \times N_0^c \times B_E(\bar x,\ep^{-1}/4)$ for $i=1,2$, where $N_0^c=N_1 \times \cdots \times N_s$. Considering a geodesic segment $\gamma$ contained in $N_0^c$ that connects $b_1$ to $b_2$, it is straightforward to see that $N_{0,q_1} \circeq N_{0,q_2}$.

Next, suppose $\pi(q_i)=x_i \in B_E(\bar x,\ep^{-1}/4)$ with $x_1 \ne x_2$. Choose a geodesic segment $\gamma_1(t)$ in $\C^{m-k}$ such that $\gamma_1(0)=x_1$ and $\gamma_1(1)=x_2$. Clearly, $\gamma_1$ is contained in $B_E(\bar x,\ep^{-1}/4)$. We can then lift $\gamma_1$ to a curve $\tilde \gamma_1$ in $U'$ such that $\tilde \gamma_1(0)=q_1$. Setting $\tilde \gamma_1(1)=q_0$, it follows that $N_{0,q_1} \circeq N_{0,q_0}$, and hence $N_{0,q_1} \equiv N_{0,q_2}$.
\end{proof}

\begin{defn} \label{def:close2}
Let $(M_i,g_i,J_i,x_i),\,i=1,2$ be two pointed K\"ahler manifolds. We say $(M_1,g_1,J_1,x_1)$ is said to be $\ep$-close to $(M_2,g_2,J_2,x_2)$ if there exist open neighborhoods $U_i$ such that $B_{g_i}(x_i,\ep^{-1}-\ep) \subset U_i \subset \bar U_i \subset B_{g_i}(x_i,\ep^{-1})$ for $i=1,2$. Moreover, there exists a diffeomorphism $\varphi:U_1 \to U_2$ such that $\varphi(x_1)=x_2$ and
\begin{align*} 
\sup_{ U_1}\lc |\varphi^* J_2-J_1|^2+ |\varphi^* g_2-g_1|^2+\sum_{i=1}^{[\ep^{-1}]} |\na^i_{g_1} (\varphi^*g_2)|^2\rc \le \ep^2.
\end{align*} 
\end{defn}

Suppose a K\"ahler manifold $(M,g,J,p) $ is $\ep$-close to $\lc N^k \times \C^{m-k}, \bar g, \bar J, \bar q \rc \in \mathcal N(m,\kappa)$. Morever, we denote their K\"ahler forms by $\omega$ and $\bar \omega$, respectively. From the diffeomorphism $\varphi$ in Definition \ref{def:close2}, we define from \eqref{E303ba}
\begin{align} 
U=\varphi^{-1}(U'), \quad V=\varphi^{-1}(V') \quad \text{and} \quad W=\varphi^{-1}(W')
\label{E304a}
\end{align} 

For any $p \in V$, there exists a submanifold $\Sigma_p$ defined as $\Sigma_p=\varphi^{-1}(N_{0,\varphi(p)})$. Now, we prove

\begin{lem}\label{lem:vol}
There exists a function $\delta_2=\delta_2(\ep,\kappa,m)>0$ with $\lim_{\ep \to 0} \delta_2(\ep,\kappa,m)=0$ such that for any $p \in V$,
\begin{align*} 
\left| \int_{\Sigma_{p}} \frac{\omega^L}{L!} -|N_0|_{\bar g} \right | \le \delta_2.
\end{align*} 
Here, $L$ is the dimension of $N_0$.
\end{lem}
\begin{proof}
It follows from a standard fact in K\"ahler geometry that for any $q \in V'$,
\begin{align*} 
\int_{N_{0,q}} \frac{\bar \omega^L}{L!}=|N_0|_{\bar g},
\end{align*} 
since $N_{0,q}$ is a complex submanifold. Now the conclusion follows immediately from Definition \ref{def:close2} and the fact that $\omega(\cdot,\cdot)=g(J\cdot, \cdot)$.
\end{proof}

It is clear from Lemma \ref{lem:equi} that $\Sigma_{p_1} \equiv \Sigma_{p_2}$ for any $p_1,p_2 \in V$. Moreover, we have

\begin{lem}\label{lem:vol}
With the above assumptions, for any $p_1,p_2 \in V$,
\begin{align*} 
\int_{\Sigma_{p_1}} \frac{\omega^L}{L!} = \int_{\Sigma_{p_2}} \frac{\omega^L}{L!}.
\end{align*} 
\end{lem}
\begin{proof}
We only need to prove the case when $\Sigma_{p_1} \circeq \Sigma_{p_2}$. From our definition, there exists an embedded submanifold $T \subset U$ such that $\partial T=\Sigma_{p_1} \cup \Sigma_{p_2}$ and $T$ is diffeomorphic to $\Sigma_{p_1} \times [0,1]$. From Stokes' theorem, and since $\omega$ is closed, we conclude 
\begin{align*} 
\int_{\Sigma_{p_2}} \omega^L-\int_{\Sigma_{p_1}} \omega^L=\int_{T} d\omega^L=0.
\end{align*} 
\end{proof}

Next, we have the following lemma similar to Lemma \ref{lem:close}. In the proof, we add the subscript $i$ to $U,V,W$, etc., to denote the corresponding sets in different charts.

\begin{lem}\label{lem: extend}
For any $m$ and $\kappa>0$, there exists a number $\ep_1=\ep_1(m,\kappa)>0$ satisfying the following property.

Let $(M,g,J)$ be a K\"ahler manifold. Suppose $(M,g,J,p_i) $ is $\ep$-close to $\lc N_i^{k_i} \times \C^{m-k_i}, \bar g_i, \bar J_i, \bar q_i \rc \in \mathcal N(m,\kappa)$ for $i=1,2$, such that $W_1 \cap W_2 \ne \emptyset$ and $\ep \le \ep_1$. Assume the decomposition of $N_i$ is given by
\begin{align*} 
N_i=N_{i,0} \times N_{i,1} \times \cdots \times N_{i,s_i}.
\end{align*}
Then, there exists $0 \le j \le s_2$ such that $N_{2,j}$,  after rescaling by a constant close to $1$, is isometric to $N_{1,0}$.
\end{lem}
\begin{proof}
From the assumption, there exists a point $q \in W_1 \cap W_2$. If we set $\psi=\varphi_2 \circ \varphi^{-1}_1$, then by using the local coordinates around $q$, we have
\begin{align*} 
\psi: N_{1,0}\times N^c_{1,0} \times B_{E}(0, \ep^{-1}/10) \longrightarrow N_2 \times \C^{m-k_2}.
\end{align*}
By the same argument as in the proof of Lemma \ref{lem:close}, if $\ep$ is sufficiently small, the map $\psi$ is a smooth embedding, which is almost isometric. Therefore, there exists a factor $N_{2,j}$ in the decomposition of $N_2$ which is almost isometric to $N_{1,0}$.
\end{proof}

By reordering the decomposition, we assume $N_{2,j}=N_{2,0}$. As discussed above, for any $p \in V_i$, there exists a submanifold $\Sigma_{i,p}$ containing $p$, which is a copy of $N_{i,0}$ for $i=1,2$. From Lemma \ref{lem: extend}, we know that $\Sigma_{1,q}$ is a small pertubation of $\Sigma_{2,q}$. Since $q \in W_1 \cap W_2$, one can find a point $q_1$ close to $q$ such that $\Sigma_{1,q_1} \circeq \Sigma_{2,q}$, in which the connecting manifold $T$ is easily constructed by the normal coordinates of $\Sigma_{2,q}$. Therefore, by the same proof of Lemma \ref{lem:vol}, we have the following result.
\begin{lem}\label{lem:vol2}
With the above assumptions, for any $p_i \in V_i$ for $i=1,2$, we have
\begin{align*} 
\int_{\Sigma_{1,p_1}} \frac{\omega^L}{L!} = \int_{\Sigma_{2,p_2}} \frac{\omega^L}{L!} .
\end{align*} 
\end{lem}

Now we prove the following proposition.

\begin{prop} \label{prop:cano}
For any $m$ and $\kappa>0$, there exists a number $\ep_2=\ep_2(m,\kappa)>0$ satisfying the following property.

Suppose $(M^m,g,J,p_0)$ is a complete K\"ahler manifold such that $R(p_0)=1$ and for any $p \in M$ with $R(p) \ge 10$, $(M^m,g_p,J,p)$ is $\ep$-close to an element in $\mathcal N(m,\kappa)$ and $\ep \le \ep_2$, where $g_p=R(p)g$. Then for all $z \in M$,
\begin{align*} 
R(z) \le 20
\end{align*} 
\end{prop}
\begin{proof}
In the proof, $\delta=\delta(\ep)$ represents a positive function such that $\lim_{\ep \to 0} \delta (\ep)=0$ and $\delta$ may be different line by line.

We prove the conclusion by contradiction. Suppose there exists $q_0 \in M$ with $R(q_0) >20$. We consider a geodesic segment $\gamma(t)$ with $\gamma(0)=q_0$ and $\gamma(1)=p_0$. From the continuity, there exists an interval $[a,b] \subset [0,1]$ such that $R(\gamma(a))=20$, $R(\gamma(b))=10$ and $R(\gamma(t)) \in [10,20]$ for any $t \in [a,b]$.

By our assumption, for any $t \in [a,b]$, $(M^m,g_{\gamma(t)},J,\gamma(t))$ is $\ep$-close to an element in $\mathcal N(m,\kappa)$ and hence the corresponding sets $U,V, V$ and $W$ are defined. From the compactness, there exists a sequence $\{t_i\}_{i=1}^K \subset [a,b]$ with $p_i=\gamma(t_i)$ satisfying the following statements.
\begin{enumerate}[label=(\roman*)]
\item $p_1=\gamma(a)$ and $p_K=\gamma(b)$. 

\item $(M,g_{p_i},J,p_i) $ is $\ep$-close to $\lc N_i^{k_i} \times \C^{m-k_i}, \bar g_i, \bar J_i, \bar q_i \rc \in \mathcal N(m,\kappa)$ for all $1 \le i \le K$.

\item $W_i \cap W_{i+1} \ne \emptyset$ for any $1 \le i \le K-1$.
\end{enumerate}

From (ii) and (iii), we conclude that
\begin{align*} 
\left| \frac{R(x)}{R(y)}-1 \right| \le \delta(\ep)
\end{align*} 
for any $x \in U_i$ and $y \in U_{i+1}$. Hence $g_{p_i}$ is $\delta$-close to $g_{p_{i+1}}$ in $C^{[\delta^{-1}]}$-topology, on $U_i \cap U_{i+1}$. As discussed above, there exists an irreducible compact Hermitian symmetric space $N_0=N_0^L$ satisfying
\begin{align} 
c_1 \le |N_0|_{\bar g} \le c_2 \label{E304aa}
\end{align} 
for two positive constants $c_1$ and $c_2$ depending only on $m$ and $\kappa$, see \eqref{E303bb}. In addition, for any point $p \in V_1$, there exists a submanifold $\Sigma_{1,p}$ containing $p$, which is a copy of $N_0$. Moreover, it follows from Lemma \ref{lem:vol} that
\begin{align*} 
\left| R^L(p_1)\int_{\Sigma_{1,p}} \frac{\omega^L}{L!} -|N_0|_{\bar g} \right | \le \delta. 
\end{align*} 
From Lemma \ref{lem: extend}, for any $p \in V_2$, we obtain a submanifold $\Sigma_{2,p}$ containing $p$, which is also a copy of $N_0$. By Lemma \ref{lem:vol2}, we have
\begin{align*} 
\int_{\Sigma_{1,p}} \frac{\omega^L}{L!} = \int_{\Sigma_{2,q}} \frac{\omega^L}{L!} 
\end{align*} 
for any $p \in V_1$ and $q \in V_2$. By iteration, we can construct $\Sigma_{i,p}$ for any $p \in V_{i}$ and $1 \le i \le K$. Moreover, we have
\begin{align} 
\left| R^L(p_i)\int_{\Sigma_{i,p}} \frac{\omega^L}{L!} -|N_0|_{\bar g} \right | \le \delta \label{E304a}
\end{align} 
for any $p \in V_{i}$ and $1 \le i \le K$. In addition,
\begin{align} 
\int_{\Sigma_{i,p}} \frac{\omega^L}{L!} = \int_{\Sigma_{i+1,q}} \frac{\omega^L}{L!} \label{E304b}
\end{align} 
for any $p \in V_i$, $q \in V_{i+1}$ and $1 \le i \le K-1$. Now we fix $x_1 \in V_1$ and $x_K \in V_K$. Combining \eqref{E304aa}, \eqref{E304a} and \eqref{E304b}, we have
\begin{align*} 
(20^L-10^L)A \le 2\delta \quad \text{and} \quad 10^LA \ge c_1-\delta,
\end{align*} 
where
\begin{align*} 
A= \int_{\Sigma_{1,x_1}} \frac{\omega^L}{L!} = \int_{\Sigma_{K,x_K}} \frac{\omega^L}{L!}.
\end{align*} 
Therefore, we obtain a contradiction if $\ep$ is sufficiently small.
\end{proof}

The same proof of Proposition \ref{prop:cano} yields the following more general result.

\begin{thm} 
For any $m$, $\kappa>0$ and $\delta>0$, there exists a number $\bar \ep=\bar \ep(m,\kappa,\delta)>0$ satisfying the following property.

Let $(M^m,g,J)$ is a complete K\"ahler manifold such that for any $p \in M$ with $R(p)\ge 1$, $(M^m,g_p,J,p)$ is $\ep$-close to an element in $\mathcal N(m,\kappa)$ and $\ep \le \bar \ep$, where $g_p=R(p)g$. If there exists $p_0 \in M$ with $R(p_0) \ge 2$, then for all $z \in M$,
\begin{align*} 
\left| \frac{R(z)}{R(p_0)}-1 \right| \le \delta.
\end{align*} 
\end{thm}

\begin{proof}
The proof follows verbatim from Proposition \ref{prop:cano} by replacing $20$ with $\max\{R(z),R(p_0)\}$ and $10$ with $(1-\delta)\max\{R(z),R(p_0)\}$.
\end{proof}

Now, we can prove Theorem \ref{thm:main1} following the same argument in \cite[Section $5$]{CL20}. We first recall the following result from \cite[Corollary $2.1$ (b)]{N05}, which originates from \cite[Corollary $11.6$]{Pe1}.

\begin{lem} \label{lem:pe}
For every $w>0$, there exist constants $C = C(w)<\infty$ and $\tau=\tau(w)>0$ with the following properties.
Let $(M^m,g(t))_{t\in [-T,0]}$ be a (possibly incomplete) K\"ahler Ricci flow solution with $BK \ge 0$. Suppose $B_{g(0)}(x_0,r_0)$ is compactly contained in $M$ such that $|B_{g(0)}(x_0,r_0)| \ge w r_0^{2m}$ and $T\ge 2 \tau r_0^2$. Then
\begin{align*}
R(x,t) \le C r_0^{-2}
\end{align*}
for $(x,t) \in B_{g(0)}(x_0,r_0/4) \times [-\tau r_0^2,0]$.
\end{lem}

Next, we have the following canonical neighborhood theorem.
\begin{thm} \label{thm:cano2}
Let $(M^m,g(t))_{t \in [0,2]}$ be a complete noncompact $\kappa$-noncollapsed K\"ahler Ricci flow solution with $BK \ge 0$. For any $\ep>0$, there exists a small number $\bar r>0$ satisfying the following property.

Suppose $(\bar x, \bar t) \in M \times [1,2]$ and $R(\bar x, \bar t)=r^{-2} \ge {\bar r}^{-2}$, then after rescaling the metric by the factor $r^{-2}$, the parabolic neighborhood $B_{g(\bar{t})}(\bar{x},\ep^{-1} r) \times [\bar{t}-\ep^{-1} r^2,\bar{t}]$ is $\ep$-close in $C^{[\ep^{-1}]}$-topology to a $\kappa$-solution.
\end{thm}
\begin{proof}
The proof is similar to the proof of \cite[Theorem $5.4$]{CL20}, and we sketch it for readers' convenience.

Assume that there exists an $\bar \ep>0$ such that the conclusion does not hold for a sequence $(x_i,t_i) \in M \times [1,2]$ with $Q_k=R(x_k,t_k) \to \infty$. By a point-picking argument, we can assume that for any $A>0$ and any $(y,t) \in B_{g(t_k)}(x_k,AQ_k^{-1/2}) \times [t_k-AQ_k^{-1},t_k]$ with $R(y,t) \ge 2Q_k$, the conclusion of the theorem holds.

Next, we consider the spacetime limit of $(M,g_k(t),x_k)$ for $t \le 0$, where $g_k(t)=Q_kg(Q_k^{-1}t+t_k)$. From Lemma \ref{lem:pe}, one can show that the limit $(M_{\infty},g_{\infty},x_{\infty})$ of $(M,g_k(0),x_k)$ is a complete K\"ahler manifold with $BK \ge 0$ and is $\kappa$-noncollapsed, see \cite[Theorem $5.4$, Step 2]{CL20} for details.

By our assumption, $(M_{\infty},g_{\infty},x_{\infty})$ satisfies the assumption of Proposition \ref{prop:cano}, if $\bar \ep$ is sufficiently small. Therefore, we conclude that he curvature of the limit $(M_{\infty},g_{\infty},x_{\infty})$ is uniformly bounded. Now, $(M_{\infty},g_{\infty}(t))$ can be extended backward to an ancient solution with uniformly bounded curvature, see \cite[Theorem $5.4$, Step 4]{CL20}. In other words, the limit $(M_{\infty},g_{\infty}(t))_{t \in (-\infty,0]}$ is a $\kappa$-solution to the K\"ahler Ricci flow.

In sum, we obtain a contradiction, and the proof is complete.
\end{proof}

Now, we can prove the boundedness of the curvature.

\begin{prop} \label{prop:bdd}
Let $(M^m,g(t))_{t \in (-\infty,0]}$ be a $\kappa$-noncollapsed, noncompact, complete ancient solution to the K\"ahler Ricci flow with $BK \ge 0$. Then the curvature of $(M,g(t))_{t \in (-\infty,0]}$ is uniformly bounded. 
\end{prop}
\begin{proof}
The proof is similar to \cite[Proposition $5.6$]{CL20}, and we sketch it for readers' convenience.

From Theorem \ref{thm:pic2}, we know that $(M,g(t))$ has weakly PIC$_2$. We first prove that for any $t_0 \le 0$, the curvature at the time $t_0$ is uniformly bounded. Otherwise, there exists a sequence $p_i$ such that $Q_i=R(p_i,t_0) \to \infty$. By applying Theorem \ref{thm:cano2} on $M \times [t_0-2,t_0]$, we conclude that $(M,Q_ig(t_0),p_i)$ converges smoothly to a $\kappa$-solution. However, this contradicts Proposition \ref{prop:cano}.

Since $(M,g(t))$ has bounded curvature on each time slice, it has bounded curvature on each compact time interval by \cite[Lemma $5.5$]{CL20}. Therefore the trace Harnack inequality holds, see \cite{Cao92} and hence the curvature is uniformly bounded since $R$ is nondecreasing along $t$.
\end{proof}

\emph{Proof of Theorem \ref{thm:main1}}: Theorem \ref{thm:main1} follows immediately from Theorem \ref{thm:ka} and Proposition \ref{prop:bdd}.

From \cite[Theorem $1.3$ (i)]{LN20}, we know that any complete ancient solution to the K\"ahler Ricci flow with nonnegative orthogonal bisectional curvature automatically has $BK \ge 0$. Moreover, a K\"ahler manifold has nonnegative orthogonal bisectional curvature if it has weakly PIC, see \cite[Proposition $9.18$]{Bre10b}. Therefore, the following corollary is immediate from Theorem \ref{thm:main1}.

\begin{cor}
Let $(M^m,g(t))_{t \in (-\infty,0]}$ be a $\kappa$-noncollapsed, complete ancient solution to the K\"ahler Ricci flow with nonnegative orthogonal bisectional curvature or weakly \emph{PIC}. Then it is isometrically biholomorphic to a finite quotient of $N^k \times \C^{m-k}$, where $N$ is a compact Hermitian symmetric space.
\end{cor}

\section{Further discussion}
In this section, we propose the following conjecture, of which Theorem \ref{thm:main2} is a special case.

\begin{conj}
Let $(M^n,g(t))_{t \in (-\infty,0]}$ be a $\kappa$-noncollapsed, Type-I, irreducible, complete ancient solution to the Ricci flow. Then it is isometric to a Ricci shrinker, up to scaling.
\end{conj}

\vskip10pt

Yu Li, Institute of Geometry and Physics, University of Science and Technology of China, No. 96 Jinzhai Road, Hefei, Anhui Province, 230026, China; Hefei National Laboratory, No. 5099 West Wangjiang Road, Hefei, Anhui Province, 230088, China; E-mail: yuli21@ustc.edu.cn. \\

\end{document}